\documentclass[a4paper,12pt]{amsart}
\usepackage[left=2cm,top=2cm,right=2cm,bottom=2cm]{geometry}
\usepackage{amsmath}
\usepackage{amssymb}
\usepackage{amsfonts}
\usepackage{amsthm}
\usepackage{tikz}

\usepackage{amsbsy}

\usepackage{stmaryrd}

\newtheorem{thm}{Theorem}[section]
\newtheorem{lemma}[thm]{Lemma}

\newtheorem{cor}[thm]{Corollary}

\newtheorem{prob}[thm]{Problem}

\theoremstyle{definition}

\newtheorem{rem}[thm]{Remark}

\theoremstyle{definition}

\newcommand{\rank}{\ensuremath{\operatorname{rank}}}

\newcommand{\norm}[1]{{\left\|{#1}\right\|}}

\newcommand{\set}[1]{{\left\{{#1}\right\}}}

\newcommand{\al}{\ensuremath{\alpha}}
\newcommand{\be}{\ensuremath{\beta}}

\newcommand{\de}{\ensuremath{\delta}}

\newcommand{\eps}{\ensuremath{\varepsilon}}
\newcommand{\ze}{\ensuremath{\zeta}}

\newcommand{\La}{\ensuremath{\Lambda}}

\newcommand{\om}{\ensuremath{\omega}}

\newcommand{\bears}{\begin{eqnarray*}}
\newcommand{\eears}{\end{eqnarray*}}

\newcommand{\mc}[1]{\ensuremath{\mathcal{#1}}}

\newcommand{\ZZ}{\ensuremath{\mathbb{Z}}}

\newcommand{\QQ}{\ensuremath{\mathbb{Q}}}
\newcommand{\RR}{\ensuremath{\mathbb{R}}}

\newcommand{\TT}{\ensuremath{\mathbb{T}}}

\newcommand{\CC}{\ensuremath{\mathbb{C}}}

\newcommand{\sm}{\ensuremath{\setminus}}
\newcommand{\ssq}{\ensuremath{\subseteq}}

\newcommand{\bs}[1]{\ensuremath{\mathbf{#1}}}

\newcommand{\bsm}{\ensuremath{\boldsymbol}}

\numberwithin{equation}{section}

\title{Rotations by roots of unity and Diophantine approximation}

\author{Romanos-Diogenes Malikiosis} 
\thanks{The author is supported by a Postdoctoral Fellowship from Humboldt Foundation.}
\address{Technische Universit\"at Berlin, Institut f\"ur Mathematik,
Sekretariat MA 4-1,
Stra{\ss}e des 17. Juni 136,
D-10623 Berlin, Germany}
\email{malikios@math.tu-berlin.de}

\begin{document}
 
 \begin{abstract}
  For a fixed integer $n$, we study the question whether at least one of the numbers $\Re X\om^k$, $1\leq k\leq n$, is $\eps$-close to an integer, 
  for any possible value of $X\in\CC$, where
  $\om$ is a primitive $n$th root of unity. It turns out that there is always a $X$ for which the above numbers are concentrated around $1/2\bmod1$. The distance from
  $1/2$ depends only on the local properties of $n$, rather than its magnitude. This is directly related the so--called ``pyjama'' problem which was
  solved recently.
 \end{abstract}

 \maketitle

 \bigskip
 \bigskip
\section{Introduction}
\bigskip

Let $n$ be a positive integer and define $\om=e^{2\pi i/n}$. For a fixed arbitrarily small value of $\eps>0$ it could be expected that for large values of $n$,
there is at least one among the 
numbers $\Re X\om^k$, $1\leq k\leq n$ who is within $\eps$ distance from an integer, for every choice of $X\in\CC$.
However, it turns out that there are values of $X$ for which all numbers $\set{\Re X\om^k}$, where $\set{x}$ denotes the fractional part of $x$, are close to
$1/2$, in particular they all fall within the interval $[1/3,2/3]$. Furthermore, we find out that they could be confined in a possibly smaller interval; the length of
the interval around $1/2$ does not depend on the magnitude of $n$ rather on its local properties. In particular, the main result of this paper is the following:

\bigskip

\begin{thm}\label{main}
 With notation as above, let $n>1$ be an integer and $\eps>0$. Then
 there is $X\in \CC$ such that $\set{\Re X\om^k}\in(\eps,1-\eps)$ for all $1\leq k\leq n$, 
 if and only if 
 \begin{enumerate}
  \item $\eps<\frac{p-1}{2p}$, when $n$ is not a power of $2$ and $p$ is the smallest odd prime dividing $n$.
  \item $\eps<1/2$, otherwise.
 \end{enumerate}
\end{thm}

\bigskip

Phrased differently, this could be considered as the cyclotomic case of the ``pyjama'' problem \cite{IKM}, which states:

\bigskip
\begin{prob}\label{pyjama}
 Let $E_{\eps}$ denote the set of 
equidistant vertical stripes on $\CC$ of width $2\eps$:
\[\{z\in\CC|\Re z\in\ZZ+[-\eps,\eps]\}.\]
Can we cover the entire plane by 
finitely many rotations of $E_{\eps}$, for arbitrarily small $\eps>0$?
\end{prob}
\bigskip

If $X\notin E_{\eps}$, then $\set{\Re X}\in(\eps,1-\eps)$, and we can easily deduce that if $X$ does not belong to any rotation of $E_{\eps}$ under the set of
modulus one complex numbers $\Theta=\set{\theta_1,\dotsc,\theta_n}$, then $\set{\Re X\theta_k^{-1}}\in(\eps,1-\eps)$ for all $k$. 
The cyclotomic case that is treated herein, considers the $n$-th roots of unity
\[\Theta=\set{e^{2\pi i k/n}|1\leq k \leq n}.\]
As Theorem \ref{main} states, all numbers $\Re X\om^k$ are far from being integers for some choices of $X$; 
this was already noted in \cite{MMR}, but the proof was omitted as it didn't contribute to the solution of Problem \ref{pyjama}.
So, what happens for other choices of angles? Taking 
$\theta_k$ to be $\QQ$-linearly independent also fails \cite{MMR}, since in that case there are choices of 
$X$ such that $\set{\Re X\theta_k^{-1}}\approx 1/2$ for
all $k$. On the other extreme, if the dimension of the $\QQ$-vector space spanned by $\theta_k$ is as small as possible, then it is equal to  $2$, and all these numbers belong to the
same imaginary quadratic field\footnote{Assuming that one rotation is the identity, or equivalently, $\theta_i=1$ for some $i$}.
However, in this case as well, there are always $X\in\CC$ such that $\set{\Re X\theta_k^{-1}}\in[1/3,2/3]$ for all $1\leq k\leq n$
\cite{MMR}. The pyjama
problem was solved in the affirmative \cite{M}, where for every $\eps>0$ an appropriate rotation set was constructed that belongs to a biquadratic extension. In the current setting,
it is proven in \cite{M} that for every $\eps>0$ there are positive integers $n$, $N$, such that for every $X\in\CC$ there is some $\theta_k\in\Theta'(n,N)$ satisfying
$\set{\Re X\theta_k^{-1}}\in[0,\eps)$ or $(1-\eps,1)$. This set $\Theta'(n,N)$ can be written as
\[\Theta'(n,N)=\Theta_N\cup \ze_1^{(n)}\Theta_N\cup\ze_{2}^{(n)}\Theta_N,\]
where $\ze_1^{(n)}$, $\ze_2^{(n)}$ unit complex numbers satisfying
\[n(1+\ze_1^{(n)})=\ze_2^{(n)},\]
while
\[\Theta_N=\set{\theta_5^r\theta_{13}^s|0\leq r, s\leq N},\]
where 
\[\theta_5=\frac{-3+4i}{5},\	\theta_{13}=\frac{-5+12i}{13}.\]
The notation is justified by the fact that $\theta_5=\mc P_5/\overline{\mc P_5}$ and $\theta_{13}=\mc P_{13}/\overline{P_{13}}$, where $\mc P_5=1+2i$ and
$\mc P_{13}=2+3i$, the primes in $\ZZ[i]$ that divide $5$ and $13$, respectively\footnote{As it is pointed out in \cite{M}, any primes of the form $1\bmod4$ would
yield this result.}. Apparently, all elements of $\Theta'(n,N)$ belong to the biquadratic extension $\QQ(i,\sqrt{4n^2-1})$.

The main tools for the proof of Theorem \ref{main} are Kronecker's approximation theorem \cite{K}, which is used to characterize the set of the $n$-tuples 
$(\Re X\theta_1^{-1},\dotsc,\Re X\theta_n^{-1})$ in Section 2, and a theorem proved independently by R\'edei \cite{R1,R2}, de Brujin \cite{deB}, and Schoenberg \cite{Sch}
regarding vanishing $\ZZ$-linear combinations of $n$th roots of unity; the proof of Theorem \ref{main} is in Section 3.

\bigskip
\bigskip

\section{Structure of the set of all $(\set{\Re X\al_1},\dotsc, \set{\Re X\al_n})$}

\bigskip
Denote by  $\TT^n$ the $n$-dimensional torus which is isomorphic to $[0,1)^n$ under the appropriate topology and group structure.  In this section we will 
investigate the structure of the subset of all $n$-tuples of the form $(\set{\Re X\al_1}, \set{\Re X\al_2},\dotsc, \set{\Re X\al_n})$ over all $X\in\CC$,
for some vector $\bsm\al=(\al_1,\dotsc,\al_n)\in\CC^n$. It is natural to work with the periodization of this set in $\RR^n$, which is
\[\set{(\Re X\al_1,\dotsc,\Re X\al_n)|X\in\CC}+\ZZ^n,\]
or simpler,
\[\ZZ^n+\RR\Re\bsm\al+\RR\Im\bsm\al,\]
where $\Re\bsm\al=(\Re\al_1,\dotsc,\Re\al_n)$ and $\Re S=\set{\Re s|s\in S}$, for any $S\ssq\CC$ (similarly for $\Im$).
Define
\[\La_{\bsm{\al}}:=\{\bsm{l}=(l_1,\dotsc,l_n)\in\ZZ^n|\bsm{l}\cdot\bsm{\al}=0\}.\]
We can simply write
\[\La_{\bsm{\al}}=\ZZ^n\cap\bsm\al^{\perp}.\]
Next, define $V_{\bsm{\al}}:=\La_{\bsm{\al}}\otimes_{\ZZ}\RR$, and $\La_{\bsm{\al}}^{*}$ the dual lattice of $\La_{\al}$ in
$V_{\bsm{\al}}$, with inner product inherited from $\RR^n$, and let
\[\mathcal{E}_{\bsm{\al}}:=\La_{\bsm{\al}}^{*}\oplus V_{\bsm{\al}}^{\perp}.\]
We have
\[\mathcal{E}_{\bsm{\al}}=\{\bsm{\xi}\in\RR^n|\bsm{l}\cdot\bsm{\xi}\in\ZZ,\forall\bsm{l}\in\La_{\bsm{\al}}\}.\]
The following Lemma is crucial towards the proof of Theorem \ref{main}; it follows from Kronecker's approximation theorem.

\bigskip
\begin{thm}[Kronecker, \cite{K}]
 Let $\bsm\al_1,\dotsc,\bsm\al_k,\bsm\be\in\RR^n$. For every $\eps>0$ there are integers $q_1,\dotsc,q_k$ and a vector $\bsm p\in\ZZ^n$ such that
 \[\norm{\sum_{i=1}^k q_i\bsm\al_i-\bsm p-\bsm\be}_{\infty}<\eps,\]
 if and only if for every $\bsm r\in\ZZ^n$ with $\bsm r\cdot\bsm\al_i\in\ZZ$ for all $1\leq i\leq k$ we also have $\bsm r\cdot\bsm\be\in\ZZ$.
\end{thm}
\bigskip

Taking the coordinates of the above vectors $\bmod1$, the above theorem gives a characterization of the closure of the subgroup of $\TT^n$ generated by
$\bsm\al_i\bmod\ZZ^n$, $1\leq i\leq k$. In particular, we get the following lemma.

\bigskip
\begin{lemma}\label{mainlemma}
Let $\bsm{\al}\in\CC^n$ and $\mathcal{E}_{\bsm{\al}}$ as above. Then, $\mathcal{E}_{\bsm{\al}}$ is the 
closure of $\ZZ^n+\RR\Re\bsm{\al}+\RR\Im\bsm{\al}$. 
\end{lemma}
\bigskip

\begin{proof}
 Inclusion is obvious: let $\bsm{\xi}\in\ZZ^n+\RR\Re\bsm{\al}+\RR\Im\bsm\al$ be arbitrary. So, $\bsm{\xi}=\bsm{n}+x\Re\bsm{\al}+y\Im\bsm\al$ for some
$\bsm{n}\in\ZZ^n$ and $x,y\in\RR$. Now, take $\bsm{l}\in\La_{\bsm{\al}}$ be arbitrary. We have
\[\bsm{l}\cdot\bsm{\xi}=\bsm{l}\cdot\bsm{n}+x(\bsm{l}\cdot\Re\bsm{\al})+y(\bsm{l}\cdot\Im\bsm\al)=\bsm{l}\cdot\bsm{n}\in\ZZ,\]
since $\La_{\bsm{\al}}\ssq\ZZ^n$ by definition. Hence, $\xi\in\mathcal{E}_{\bsm{\al}}$.

Now let $\bsm{x}=(x_1,\dotsc,x_n)\in\mathcal{E}_{\bsm{\al}}$ be arbitrary. We will approximate $\bsm{x}$ by elements of
$\ZZ^n+\RR\Re\bsm{\al}+\RR\Im\bsm\al$ as close as possible. First, we assume $\al_n\neq0$. Without loss of generality, we may assume
 $\al_n=1$; the sets $\mathcal{E}_{\bsm{\al}},\La_{\bsm{\al}},V_{\bsm{\al}}$ remain unchanged if we multiply $\bsm{\al}$ by
 any nonzero complex number. It suffices to prove that the double sequence 
\[(x_n+k)\Re\bsm{\al}+m\Im\bsm\al-\bsm{x}=((x_n+k)\Re\al_1+m\Im\al_1-x_1,\dotsc,(x_n+k)\Re\al_{n-1}+m\Im\al_{n-1}-x_{n-1},k), k,m\in\ZZ\]
has terms arbitrarily close to $\ZZ^n$, or equivalently, the double sequence
\[k(\Re\al_1,\dotsc,\Re\al_{n-1})+m(\Im\al_1,\dotsc,\Im\al_{n-1}), k,m\in\ZZ\]
has terms arbitrarily close to $(x_1-x_n\Re\al_1,\dotsc,x_{n-1}-x_n\Re\al_{n-1})+\ZZ^{n-1}$. Consider the projection $\pi:\CC^n\rightarrow\CC^{n-1}$
that ``forgets'' the last coordinate and the natural epimorphism $\varphi:\RR^{n-1}\rightarrow\RR^{n-1}/\ZZ^{n-1}=\TT^{n-1}$. Since
$\al_n=1$, we have 
\[\pi(\La_{\bsm{\al}})=\set{(l_1,\dotsc,l_{n-1})\in\ZZ^{n-1}|l_1\al_1+\dotsb+l_{n-1}\al_{n-1}\in\ZZ}.\]
By Kronecker's approximation theorem, the closure of the subgroup of $\TT^{n-1}$
generated by 
\[\pi(\Re\bsm{\al})\bmod\ZZ^{n-1}=(\Re\al_1,\dotsc,\Re\al_{n-1})\bmod\ZZ^{n-1}\]
and
\[\pi(\Im\bsm{\al})\bmod\ZZ^{n-1}=(\Im\al_1,\dotsc,\Im\al_{n-1})\bmod\ZZ^{n-1}\]
is the set of all $(a_1,\dotsc,a_{n-1})+\ZZ^{n-1}$ for which $l_1a_1+\dotsb+l_{n-1}a_{n-1}\in\ZZ$ whenever $l_1,\dotsc l_{n-1}\in\ZZ$ and 
$l_1\Re\al_1+\dotsb+l_{n-1}\Re\al_{n-1},l_1\Im\al_1+\dotsb+l_{n-1}\Im\al_{n-1}\in\ZZ$, i.~e. $(l_1,\dotsc,l_{n-1})\in\pi(\La_{\al})$. 
In other words, it suffices to prove that
\[\pi(\bsm{l})\cdot(x_1-x_n\Re\al_1,\dotsc,x_{n-1}-x_n\Re\al_{n-1})\in\ZZ\]
whenever $\bsm{l}\in\La_{\bsm{\al}}$. So, when $\bsm{l}\in\La_{\bsm{\al}}$, by definition we have $\bsm{l}\cdot\bsm{x}\in\ZZ$.
Hence, $\bsm{l}\cdot(\bsm{x}-x_n\Re\bsm{\al})\in\ZZ$. The latter is obviously equal to the desired inner product, so the proof
is finished for $\al_n\neq0$.

If $\bsm{\al}\neq\bsm{0}$, then one coordinate is different from zero, say $\al_i$; as above we may assume that $\al_i=1$ and we 
argue as before, with $\pi$ replaced by $\pi^i:\RR^n\rightarrow\RR^{n-1}$, the projection that ``forgets'' the $i$-th
coordinate. If $\bsm{\al}=\bsm{0}$, then obviously $\mathcal{E}_{\bsm{\al}}=\ZZ^n$. This completes the proof.
\end{proof}

\bigskip

\begin{cor}
When the coordinates of $\bsm{\al}$ are linearly independent over $\QQ$, then 
 the set
\[\{(\{\Re x\al_1\},\dotsc,\{\Re x\al_n\})|x\in\CC)\}\]
is dense in $\TT^n$.
\end{cor}
\bigskip

\begin{proof}
By hypothesis, we have
 \[\La_{\bsm{\al}}=V_{\bsm{\al}}=\La_{\bsm{\al}}^*=\{\bsm{0}\}\]
and
\[V_{\bsm{\al}}^{\perp}=\mathcal{E}_{\bsm{\al}}=\RR^n,\]
therefore $\ZZ^n+\RR\Re\bsm{\al}+\RR\Im\bsm\al$ is dense in $\RR^n$ by Lemma \ref{mainlemma}, which is equivalent to the desired conclusion.
\end{proof}

\bigskip
\begin{rem}\rm
 The multidimensional version of Weyl's criterion \cite{Weyl,KN} gives something stronger than Kronecker's approximation theorem, in particular that the subgroup
 of $\TT^n$ generated by $\bsm\al_i\bmod\ZZ^n$ is equidistributed in the set of all $\bsm\be\mod\ZZ^n$ with $\bsm r\cdot\be\in \ZZ$. whenever $\bsm r\cdot\bsm\al_i\in \ZZ$
 for all $i$, but this is not necessary for our purposes.
\end{rem}

\bigskip
Let $\de_{\bsm\al}$ be the $l^{\infty}$ distance from $\tfrac{1}{2}\bsm1_n$ to $\mc{E}_{\al}$, where $\bsm1_n=(1,1,\dotsc,1)$ denotes the all-$1$ vector. The above remark
yields $\de_{\bsm\al}=0$ whenever $\rank\La_{\bsm\al}=0$, i.~e. when the coordinates of $\bsm\al$ are linearly independent over $\QQ$. We can also compute $\de_{\bsm\al}$
when $\rank\La_{\bsm\al}=1$, that is, when $\La_{\bsm\al}$ is generated by a single vector $\bsm{l}=(l_1,\dotsc,l_n)$
for which $\gcd(l_1,\dotsc,l_n)=1$.

$\mathcal{E}_{\bsm{\al}}$ is then the union of the hyperplanes $H_M=\{\bsm{x}\in\RR^n|\bsm{l}\cdot\bsm{x}=M\}$, for all $M\in\ZZ$.
Then,
\[\de_{\bsm{\al}}=\inf_{M\in\ZZ}\inf_{\bsm{x}\in H_M}\lVert\bsm{x}-\tfrac{1}{2}\bs1_n\rVert_{\infty}.\]
We need the following:
\bigskip

\begin{lemma}
For each $M\in\ZZ$, we have
\[\inf_{x\in H_M}\lVert\bsm{x}-\tfrac{1}{2}\bs1_n\rVert_{\infty}=\frac{\lvert M-\tfrac{1}{2}\bsm{l}\cdot\bs1_n\rvert}{\lVert\bsm{l}\rVert_1}.\]
\end{lemma}
\bigskip

\begin{proof}
 If $\tfrac{1}{2}\bs1_n\in H_M$, then both sides are equal to $0$. We assume that $\tfrac{1}{2}\bs1_n\notin H_M$. By H\"{o}lder's extremal equality,
we have
\[\lVert\bsm{l}\rVert_1=\sup_{\bsm{x}\in\RR^n\sm\{\tfrac{1}{2}\bs1_n\}}\frac{\lvert\bsm{l}\cdot(\bsm{x}-\tfrac{1}{2}\bs1_n)\rvert}{\lVert
\bsm{x}-\tfrac{1}{2}\bs1_n\rVert_{\infty}}.\]
By homogeneity of the above norms, this supremum remains invariant if we restrict $\bsm{x}$ to $H_M$. Indeed, let
$\bsm{x}_t=\tfrac{1}{2}\bs1_n+t(\bsm{x}-\tfrac{1}{2}\bs1_n)$. Then, we observe that
\[\frac{\lvert\bsm{l}\cdot(\bsm{x}_t-\tfrac{1}{2}\bs1_n)\rvert}{\lVert\bsm{x}_t-\tfrac{1}{2}\bs1_n\rVert_{\infty}}
=\frac{\lvert t\rvert\cdot\lvert\bsm{l}\cdot(\bsm{x}-\tfrac{1}{2}\bs1_n)\rvert}{\lvert t\rvert\cdot\lVert\bsm{x}-\tfrac{1}{2}\bs1_n\rVert_{\infty}}
=\frac{\lvert\bsm{l}\cdot(\bsm{x}-\tfrac{1}{2}\bs1_n)\rvert}{\lVert\bsm{x}-\tfrac{1}{2}\bs1_n\rVert_{\infty}}\]
so this expression remains invariant on any ray emanating from $\tfrac{1}{2}\bs1_n$. It suffices to consider the rays intersecting $H_M$;
those who are parallel to $H_M$ satisfy $\bsm{l}\cdot(\bsm{x}-\tfrac{1}{2}\bs1_n)=0$ and do not contribute to the above supremum. Therefore,
\[\lVert\bsm{l}\rVert_1=\sup_{\bsm{x}\in H_M}\frac{\lvert\bsm{l}\cdot(\bsm{x}-\tfrac{1}{2}\bs1_n)\rvert}{\lVert\bsm{x}-\tfrac{1}{2}\bs1_n\rVert_{\infty}}
=\frac{\lvert M-\bsm{l}\cdot\tfrac{1}{2}\bs1_n\rvert}{\inf_{\bsm{x}\in H_M}\lVert\bsm{x}-\tfrac{1}{2}\bs1_n\rVert_{\infty}}\]
which easily yields the desired result.
\end{proof}
\bigskip

\begin{thm}\label{rankone}
 We have $\de_{\bsm{\al}}=0$ when $\lVert\bsm{l}\rVert_1$ is even, and $\de_{\bsm{\al}}=\frac{1}{2\lVert\bsm{l}\rVert_1}$ otherwise.
\end{thm}
\bigskip

\begin{proof}
 If $\lVert\bsm{l}\rVert_1$ is even, then $\bsm{l}\cdot\tfrac{1}{2}\bs1_n\in\ZZ$ and $\tfrac{1}{2}\bs1_n\in\mathcal{E}_{\bsm{\al}}$, hence $\de_{\bsm{\al}}=0$.
 Otherwise, $\bsm{l}\cdot\tfrac{1}{2}\bs1_n$ is a half integer. Using the above Lemma,
\[\de_{\bsm{\al}}=\frac{\inf_{M\in\ZZ}\lvert M-\bsm{l}\cdot\tfrac{1}{2}\bs1_n\rvert}{\lVert\bsm{l}\rVert_1}=
\frac{1}{2\lVert\bsm{l}\rVert_1}.\qedhere\]
\end{proof}
\bigskip

We now restrict to modulus one complex numbers and let $\bsm\theta=(\theta_1,\dotsc,\theta_n)\in(S^1)^n$. Consider a fixed $0<\eps<1/2$. By definition there is
some $\bsm\xi=(\xi_1,\dotsc,\xi_n)\in\mc{E}_{\bsm\theta}$ such that $\set{\xi_k}\in(\eps,1-\eps)$ for all $k$ if and only if $\de_{\bsm\theta}<\tfrac{1}{2}-\eps$.
Then, from Lemma \ref{mainlemma} we deduce the following corollary:

\bigskip
\begin{cor}\label{structure}
With notation as above, there is some $X\in\CC$ such that $\set{\Re X\theta_k}\in(\eps,1-\eps)$ for all $k$, $1\leq k\leq n$,
if and only if $\eps<\tfrac{1}{2}-\de_{\bsm\theta}$.
\end{cor}
\bigskip

\begin{rem}
 In reference to the ``pyjama'' problem and the invariance of $\La_{\bsm\theta}$, $V_{\bsm\theta}$, $\mc{E}_{\bsm\theta}$, and $\de_{\bsm\theta}$
 under complex conjugation of the coordinates of $\bsm\theta$, the above Corollary is equivalent to the following statement: {\em the union the rotations of
 $E_{\eps}$ by $\theta_1,\dotsc,\theta_n$ covers $\CC$ if and only if $\eps<\tfrac{1}{2}-\de_{\bsm\theta}$}.
\end{rem}

\bigskip
\bigskip

\section{Roots of unity and proof of main result}
\bigskip

We further restrict to rotations that correspond to roots of unity. Let $\om_n=e^{\frac{2\pi i}{n}}$ be a primitive root of unity, $\bsm\om_n=(1,\om_n,\om_n^2,\dotsc,\om_n^{n-1})$, 
and denote by $\La_n$ the
lattice $\La_{\bsm\om_n}$. Similarly, denote $\mathcal{E}_n:=\mathcal{E}_{\bsm\om_n}$ and $\de_n:=\de_{\bsm\om_n}$. When $n=p$ is a prime, then $\La_p$ has rank 
one\footnote{This follows easily from the fact that $[\QQ(\om_p):\QQ]=p-1$ for $p$ prime} and is
generated by $\bs 1_p$. Then, by Theorem \ref{rankone} we get $\de_2=0$ and
\begin{equation}\label{deltaprime}
 \de_p=\frac{1}{2p},
\end{equation}
for an odd prime $p$. Moreover, when $m$ divides $n$, the rotations of $X\in\CC$ by $m$th roots of unity is a subset those by $n$th roots of unity, hence we obtain
\begin{equation}\label{deltadiv}
 \de_m\leq\de_n,\	\text{ if }\	m|n.
\end{equation}
In order to compute $\de_n$ for all $n$, we need a nice characterization of $\La_n$, or equivalently, we need to characterize all $\ZZ$-linear combinations of $n$th roots
of unity that amount to $0$. This characterization was first given by R\'edei \cite{R1} with an incomplete proof\footnote{This is mentioned in \cite{deB}; see also \cite{LL}};
complete proofs were later given by de Brujin \cite{deB}, R\'edei \cite{R2}, and Schoenberg \cite{Sch}.

Denote by $\bs e(i,k)$ the unit vector in $\CC^k$ whose $i$th coordinate is equal to $1$. Finally, we will denote by 
$\otimes$ the usual Kronecker product, and by $\Phi_n$ the cyclotomic polynomial of order $n$.


\bigskip
\begin{lemma}[\cite{deB,R1,R2,Sch}]\label{rtsstr}
 $\La_n$ is generated by the vectors $\bs 1_p\otimes\bs e(i,\frac{n}{p})$, for all primes $p\mid n$ and $1\leq i
\leq \frac{n}{p}$.
\end{lemma}
\bigskip

\begin{rem}\rm
 Notice that the above vectors form a basis of $\La_n$ precisely when $n$ is a prime power. In the general case, the above Lemma states that all $\ZZ$-linear
 combinations of $n$th roots of unity that vanish, are essentially $\ZZ$-linear combinations of formal sums of the form
 \[\ze(1+\om_p+\dotsb+\om_{p}^{p-1}),\]
 where $\ze$ is any $n$th root of unity and $p$ is any prime divisor of $n$.
\end{rem}
\bigskip

We will now construct a point in $\mathcal{E}_n$ whose
distance from $\frac{1}{2}\bs 1_n$ is precisely $\frac{1}{2p}$, where $p$ is the smallest odd prime dividing $n$, when $n$
is not a power of $2$. In the remaining case, we will show that $\frac{1}{2}\bs 1_n\in\mathcal{E}_n$.

\bigskip
\begin{proof}[Proof of Theorem \ref{main}]
Assume first that $n$ is a power of $2$. Then, by Lemma \ref{rtsstr} the lattice $\La_n$ is generated by the vectors $\bs 1_2\otimes\bs e(i,\frac{n}{2})$
for $1\leqslant i\leqslant\frac{n}{2}$. Hence, the inner product of the vectors $\frac{1}{2}\bs 1_n$ and $\bs 1_2\otimes\bs e(i,\frac{n}{2})$
is equal to $1$ for all $i$. This clearly shows that $\frac{1}{2}\bs 1_n\in\mathcal{E}_n$, hence $\de_n=0$, as was to be shown.

Next, assume that $n$ is not a power of $2$, and let $p$ be its smallest odd prime divisor. Consider the following polynomials:
\begin{eqnarray*}
\Psi_2(X)=X-1,\	\Psi_p(X)=\Phi_p(X),\\
\Psi_q(X)=\sum_{i=0}^{(q+p)/2}X^i-\sum_{i=\frac{q+p}{2}+1}^{q-1}X^i, q>p, q\mid n.
\end{eqnarray*}
Next, define
\[T(X)=\sum_{j\geqslant0}\tau_jX^j:=\frac{X^{\frac{n}{m}-1}}{X-1}\prod_{q\mid n}\Psi_q(X^{\frac{n}{q}})\]
where $m$ denotes the square-free part of $n$. If we expand the product on the right, we get $n$ terms of the form
\[\pm X^{j+\sum_{q\mid n}j_q\frac{n}{q}}\]
where $0\leqslant j\leqslant \frac{n}{m}-1$, $0\leqslant j_q\leqslant q-1$, for all primes $q$ dividing $n$. It is not
hard to verify that the above exponents are pairwise incongruent modulo $n$. So, if we let $\rho(X)$ be the remainder of 
$T(X)$ divided by $X^n-1$, we will have
\[\rho(X)=\sum_{j=0}^{n-1}\eps_jX^j\]
where $\eps_j=\pm1$ for all $j$. Next, consider the vector
\[\bsm{\xi}=\left(\frac{p-\eps_0}{2p},\frac{p-\eps_1}{2p},\dotsc,\frac{p-\eps_{n-1}}{2p}\right).\]
We will show that $\bsm{\xi}\in\mathcal{E}_n$. Let $l$ be an arbitrary prime dividing $n$; the inner product of $\bsm{\xi}$
with $\bs 1_l\otimes\bs e(i,\frac{n}{l})$ is
\[\frac{l}{2}-\frac{1}{2p}\sum_{j\equiv i\bmod\frac{n}{l}}\eps_j=\frac{l}{2}-\frac{1}{2p}\sum_{j\equiv i\bmod\frac{n}{l}}\tau_j.\]
From the definition of $T(X)$, we must have
\[\sum_{j\equiv i\bmod{\frac{n}{l}}}\tau_jX^j=\pm X^{j+\sum_{q\mid \frac{n}{l}}j_q\frac{n}{q}}\Psi_l(X^{\frac{n}{l}})\]
where $q$ runs through the primes dividing $\frac{n}{l}$ and $j$, $j_q$ ($q\neq l$) are the unique integers satisfying
$0\leqslant j\leqslant \frac{n}{m}-1$, $0\leqslant j_q\leqslant q-1$, and
\[j+\sum_{q\mid \frac{n}{l}}j_q\frac{n}{q}\equiv i\bmod\frac{n}{l}.\]
Thus, $\sum_{j\equiv i\bmod\frac{n}{l}}\tau_j=\pm\Psi_l(1)$. If $l$ is odd, this sum is equal to $\pm p$, and if $l=2$, it is equal to
$0$. At any rate, we have $\langle\bsm{\xi},\bs 1_l\otimes\bs e(i,\frac{n}{l})\rangle\in\ZZ$, therefore by Lemma \ref{rtsstr} we
get $\bsm{\xi}\in\mathcal{E}_n$. The $l^{\infty}$ distance from $\frac{1}{2}\bs 1_n$ to $\bsm{\xi}$ is precisely $\frac{1}{2p}$,
so $\de_n\leq \frac{1}{2p}$. On the other hand, $\de_n\geq\de_p$ by \eqref{deltadiv}, and $\de_p=\frac{1}{2p}$ by 
\eqref{deltaprime}, thus $\de_n=\frac{1}{2p}$. Finally, Corollary \ref{structure} yields the conclusion of Theorem \ref{main} by observing
that $1/2-\de_n=\frac{p-1}{2p}$, when $n$ is not a power of $2$, while $1/2-\de_n=1/2$ otherwise.
\end{proof}

\end{document}